\documentclass{amsart}
\usepackage{amssymb,amsrefs,mathrsfs}

\newcommand\Z{\mathbb{Z}}
\newcommand\N{\mathbb{N}}
\newcommand\R{\mathbb{R}}

\newtheorem{lemma}{Lemma}[section]

\newtheorem{proposition}[lemma]{Proposition}
\newtheorem{theorem}[lemma]{Theorem}
\newtheorem{corollary}[lemma]{Corollary}

\newtheorem{maintheorem}{Theorem}

\newtheorem{maincorollary}[maintheorem]{Corollary}

\theoremstyle{definition}
\newtheorem{remark}[lemma]{Remark}
\newtheorem{definition}[lemma]{Definition}

\newcommand\grig{G_{012}}

\newcommand\coarse{coarse}

\begin{document}
\title[Distortion of imbeddings into metric spaces]{Distortion of imbeddings of groups of intermediate growth into metric spaces}

\author{Laurent Bartholdi}
\address{L.B.: Mathematisches Institut, Georg-August Universit\"at, G\"ottingen, Germany}
\email{laurent.bartholdi@gmail.com}

\author{Anna Erschler}
\address{A.E.: C.N.R.S., D\'epartement de Math\'ematiques, Universit\'e Paris Sud, Orsay, France}
\email{anna.erschler@math.u-psud.fr}

\date{June 23, 2014}

\thanks{The work is supported by the ERC starting grant 257110
  ``RaWG'', the ANR ``DiscGroup: facettes des groupes discrets'', the
  Centre International de Math\'ematiques et Informatique, Toulouse,
  and the Institut Henri Poincar\'e, Paris}

\begin{abstract}
  For every metric space $\mathcal X$ in which there exists a sequence
  of finite groups of bounded-size generating set that does not embed
  coarsely, and for every unbounded, increasing function $\rho$, we
  produce a group of subexponential word growth all of whose
  imbeddings in $\mathcal X$ have distortion worse than $\rho$.

  This applies in particular to any B-convex Banach space $\mathcal
  X$, such as Hilbert space.
\end{abstract}
\maketitle

\section{Introduction}
Let $G$ be a finitely generated group, and let $(\mathcal X,d)$ be a
metric space. The extent to which $G$, with its word metric, may be
imbedded in $\mathcal X$ with not-too-distorted metric is an
asymptotic invariant of $G$, introduced by Gromov
in~\cite{gromov:asympt}*{\S7.E}. The general definition of distortion
is as follows:

\begin{definition}
  Consider a $1$-Lipschitz map $\Phi:(\mathcal Y,d)\to(\mathcal X,d)$
  between metric spaces. Its \emph{distortion} is the function
  \[\rho_\Phi\colon\R_+\to\R_+,\qquad\rho_\Phi(t)=\inf_{d(y,y')\ge t}d(\Phi(y),\Phi(y')).\]
  It is the largest increasing function $\rho\colon\R_+\to\R_+$ such
  that
  \[\rho(d(y,y'))\le d(\Phi(y),\Phi(y'))\le d(y,y')\text{ for all }y\neq y'\in
  \mathcal Y.
  \]
  If $\mathcal Y$ has bounded diameter, then $\rho_\Phi(t)=+\infty$
  for all $t>\operatorname{diam}\mathcal Y$.  We say that $\Phi$ has
  distortion \emph{better than $\rho$} if $\rho_\Phi(t)>\rho(t)$ for
  all $t\in\R_+$ large enough, \emph{worse than $\rho$} if
  $\rho_\Phi(t)<\rho(t)$ for all $t\in\R_+$ large enough, and that
  $\Phi$ is a \emph{\coarse\ imbedding} if its distortion is
  unbounded\footnote{An alternative terminology is \emph{uniform
      imbedding}, which we avoid.}.

  More generally, if $(\mathcal Y_i)_{i\in I}$ is a family of metric
  spaces, a \emph{\coarse\ imbedding} is a sequence $(\Phi_i\colon
  \mathcal Y_i\to\mathcal X)$ of $1$-Lipschitz imbeddings with
  $\inf_{i\in I}\rho_{\Phi_i}(t)$ an unbounded function of $t$.
\end{definition}

\noindent Our main result is:

\begin{maintheorem}[= Theorem~\ref{thm:Edistortion}]
  Let $\mathcal X$ be a metric space, and let $(G_i)_{i\in\N}$ be a
  sequence of $d$-generated finite groups that do not imbed coarsely
  in $\mathcal X$.

  Then, for every unbounded increasing function
  $\rho\colon\R_+\to\R_+$, there exists a finitely generated group $W$
  of subexponential growth such that every imbedding of $W$ in
  $\mathcal X$ has distortion worse than $\rho$.

  Furthermore, the group $W$ contains an infinite subsequence of the
  $G_i$'s.
\end{maintheorem}

\begin{maincorollary}[= Corollary~\ref{cor:distortion}]
  For every unbounded increasing function $\rho\colon\R_+\to\R_+$ and
  for every B-convex Banach space $\mathcal X$ (e.g.\ Hilbert space;
  see~\S\ref{ss:superexpanders} for the definition) there exists a
  finitely generated group $W$ of subexponential word growth such
  that every imbedding of $W$ in $\mathcal X$ has distortion worse
  than $\rho$.
\end{maincorollary}

Furthermore, the group $W$ depends on $\rho$ but only mildly on
$\mathcal X$: for every unbounded increasing function $\rho$ there
exists a group $W$ of subexponential growth, such that every imbedding
of $W$ in a B-convex Banach space $\mathcal X$ has distortion worse
than $c\rho$ for some constant $c$; and moreover $c$ depends only on
the convexity parameters $n,\epsilon$ of $\mathcal X$, see
Remark~\ref{rem:switchquantifiers}.

Recall that a finitely generated group $G$ has \emph{subexponential
  growth} if, for every $\lambda>1$, the number of group elements that
are products of at most $n$ generators grows more slowly than the
exponential function $\lambda^n$. A group has \emph{locally
  subexponential growth} if all its finitely generated subgroups have
subexponential growth.

Arzhantseva, Dru\c tu and Sapir construct
in~\cite{arzhantseva-drutu-sapir:compression}*{Theorem~1.5}, for every
unbounded increasing function $\rho$, a group which imbeds \coarse ly
into Hilbert space, and such that all of its imbeddings into a
uniformly convex Banach space have distortion worse than $\rho$ (for a
slightly weaker definition of ``worse'' than ours).  Olshansky and
Osin construct, moreover, amenable groups with the same property
in~\cite{olshanskii-osin:qiembedding}*{Corollary~1.4}.

These examples all have exponential growth. In contrast, the main
point of our construction is to produce such groups having
subexponential growth. These are in particular the first known
examples of groups whose simple random walks have trivial Poisson
boundary and with arbitrarily bad distortion in every imbedding into
Hilbert space.

It follows from~\cite{naor-peres:embeddings}*{Theorem~1.1} by Naor and
Peres that, if an amenable group $G$ admits an imbedding with
distortion better than $n^{1/2+\epsilon}$ for some $\epsilon>0$, then
every simple random walk on $G$ has trivial Poisson boundary. Our
result shows that groups with trivial Poisson boundary for every
simple random walk may have arbitrarily bad distortion in every
imbedding into Hilbert space.

Note that groups of subexponential growth are amenable, and Bekka,
Ch\'erix and Valette show in~\cite{bekka-cherix-valette:affineactions}
that amenable groups imbed \coarse ly into Hilbert space. In fact,
their imbeddings can be shown to have distortion better than an
unbounded function which depends only on the F\o lner function.
Tessera~\cite{tessera:uniform}*{Theorem~10} gives such formul\ae\
in terms of the isoperimetric profile inside balls.

For example, consider the Grigorchuk groups $G_\omega$ of intermediate
growth, introduced in~\cite{grigorchuk:gdegree}. They admit
``self-similar random walks'' $\mu_\omega$ in the sense
of~\cite{bartholdi-erschler:boundarygrowth}*{\S6}
and~\cite{kaimanovich:munchhausen}, with asymptotic entropy
$H(\mu_\omega^n)\precsim n^{1/2}$
by~\cite{bartholdi-erschler:boundarygrowth}*{Corollary~6.3}; so their
probability of return satisfies $\mu_\omega^n(1)\succsim
\exp(-n^{1/2})$ using the general estimate
$\mu^n(1)\ge\exp(-2H(\mu^n))$, their F\o lner function satisfies
$F(n)\precsim \exp(n^2)$ using Nash inequalities
(see~\cite{woess:rw}*{Corollary~14.5(b)}, and therefore they admit
imbeddings into Hilbert space of distortion better than
$t^{1/2-\epsilon}$ for every $\epsilon>0$, by a result of
Gournay~\cite{gournay:liouville}.

On the other hand, groups of subexponential growth can have
arbitrarily large F\o lner function~\cite{erschler:piecewise}. Our
result can therefore be seen as a strengthening of this fact.

\subsection{Acknowledgments}
We are grateful to Bogdan Nica for helpful discussions on generating
sets for linear groups, and to Mikael de la Salle for explanations on
Lafforgue's construction of expanders and for
Remark~\ref{rem:switchquantifiers}.

\section{Sketch of the proof}
We apply the construction of~\cite{bartholdi-erschler:imbeddings},
which constructs, for any countable group $B$ all of whose finitely
generated subgroups have subexponential growth, a finitely generated
group of subexponential growth in which $B$ imbeds;
see~\S\ref{ss:permutational} and~\S\ref{ss:W}. We take for $B$ a
restricted direct product of finite groups $H_i$ with poor imbedding
properties, following the idea of Arzhantseva, Dru\c tu and
Sapir~\cite{arzhantseva-drutu-sapir:compression}.

We proceed in three steps: first, given a sequence of finite groups
$G_1,G_2,\dots$ with bounded number of generators, we imbed each each
$G_i$ as the derived subgroup of a finite group $H_i$ with one more
generator, in such a manner that the metric on $G_i$ is at universally
bounded distance from a particular metric on $[H_i,H_i]$, the
\emph{perfect metric} (see~\S\ref{ss:derived}).

We then show in~\S\ref{ss:inW} that an arbitrary subset of the $H_i$'s
imbeds in a finitely generated group $W$ of subexponential growth with
controlled distortion; more precisely, the distortion constants of an
imbedded $H_i$ in $W$ depends only on the previous $H_j$'s, but not on
$H_i$.

Finally, in~\S\ref{ss:inMS} we assume that the $H_i$ do not imbed
\coarse ly in a metric space $\mathcal X$. Given any unbounded
increasing function $\rho$, we select the subset of $H_i$'s
appropriately so that the distortion of $W$ in $\mathcal X$ is worse
than $\rho$.

To prove the corollary in~\S\ref{ss:superexpanders}, we choose for the
$G_i$, or even directly for the $H_i$, a family of superexpanders such
as those constructed by Lafforgue in~\cite{lafforgue:Tbanachique}.

\section{Permutational wreath products}\label{ss:permutational}
We recall briefly some notions introduced
in~\cite{bartholdi-erschler:imbeddings}. For details, we refer to that
article.

Let $G=\langle S\rangle$ be a finitely generated group acting on the
right on a set $X$. We consider $X$ as a the vertex set of a graph
still denoted $X$, with for all $x\in X,s\in S$ an edge labelled $s$
from $x$ to $xs$. We denote by $d$ the path metric on this graph.

\begin{definition}
  A sequence $(x_0,x_1,\dots)$ in $X$ is \emph{spreading} if for all
  $R$ there exists $N$ such that if $i,j\ge N$ and $i\neq j$ then
  $d(x_i,x_j)\ge R$.
\end{definition}

\begin{definition}
  A sequence $(x_i)$ in $X$ \emph{locally stabilises} if for all $R$
  there exists $N$ such that if $i,j\ge N$ then the $S$-labelled
  radius-$R$ balls centered at $x_i$ and $x_j$ in $X$ are equal.
\end{definition}

\begin{definition}
  A sequence of points $(x_i)$ in $X$ is \emph{rectifiable} if for all
  $i,j$ there exists $g \in G$ with $x_i g =x_j$ and $x_k g\ne x_\ell$
  for all $k\notin\{i,\ell\}$.
\end{definition}

\begin{definition}
  The group $G$ acting on $X$ has the \emph{subexponential wreathing
    property} if for any finitely generated group of subexponential
  growth $H$ the restricted wreath product $H\wr_X G$ has
  subexponential growth.
\end{definition}

We summarize, in the following proposition, an example of group $G$,
action on $X$ and rectifiable, separating, locally stabilizing
sequence $(x_i)$ with the subexponential wreathing property:

\begin{proposition}[\cite{bartholdi-erschler:imbeddings}*{Lemma~4.9}
  and \cite{bartholdi-erschler:permutational}*{Proposition~4.4}]\label{prop:grig}
  Let $G=\grig=\langle a,b,c,d\rangle$ denote the first Grigorchuk
  group. Recall that it acts on set of infinite sequences
  $\{\mathbf0,\mathbf1\}^\infty$ over a two-letter alphabet, which is
  naturally the boundary of a binary rooted tree. Denote by
  $X=\mathbf1^\infty\grig$ the orbit of the rightmost ray.

  Then $G$ has the subexponential wreathing property in its action on
  $X$, and the sequence $(x_i)$ defined by
  $x_i=\mathbf0^i\mathbf1^\infty$ is rectifiable, spreading and
  locally stabilizing.\qed
\end{proposition}
  
\section{Imbedding groups of locally subexponential growth in $W$}\label{ss:W}
We assume that a group $G$ acting on a set $X$, and a separating,
spreading, locally stabilizing sequence $(x_i)$ of elements of $X$
have been fixed; such data exist by Proposition~\ref{prop:grig}. We
repeat, in this Section, the contents
of~\cite{bartholdi-erschler:imbeddings}*{\S6}, since they are
fundamental to the argument.

Let $B$ be a group, and let $(b_1,b_2,\dots)$ be a sequence in $B$. We
will construct a rapidly increasing sequence $0\le n(1) < n(2) <\dots$
later; assuming this sequence given, we define $f\colon X\to B$ by
\[f(x_{n(1)})=b_1,\qquad f(x_{n(2)})=b_2,\qquad \dots,\qquad f(x)=1\text{ for other }x.
\]
We then consider the subgroup $W=\langle G,f\rangle$ of the
unrestricted wreath product $B^X\rtimes G$.

\begin{lemma}[\cite{bartholdi-erschler:imbeddings}*{Lemma~6.1}]\label{lem:contains [B,B]}
  Denote by $B_0$ the subgroup of $B$ generated by
  $\{b_1,b_2,\dots\}$. If the sequence $(x_i)$ is separating, then $W$
  contains $[B_0,B_0]$ as a subgroup.
\end{lemma}
\begin{proof}
  Without loss of generality and to lighten notation, we rename $B_0$
  into $B$. We also denote by $\iota\colon B\to B^X\rtimes G$ the
  imbedding of $B$ mapping the element $b\in B$ to the function $X\to
  B$ with value $b$ at $x_0$ and $1$ elsewhere.  We shall show that
  $W$ contains $\iota([B,B])$. For this, denote by $H$ the subgroup
  $\iota(B)\cap W$.

  We first consider an elementary commutator $g=[b_i,b_j]$. Let
  $g_i,g_j\in G$ respectively map $x_i,x_j$ to $x_0$, and be such that
  $g_i g_j^{-1}$ maps no $x_k$ to $x_\ell$ with $k\neq\ell$, except
  for $x_i g_i g_j^{-1}=x_j$. Consider $[f^{g_i},f^{g_j}]\in W$; it
  belongs to $B^X$, and has value $[b_i,b_j]$ at $x_0$ and is trivial
  elsewhere, so equals $\iota(g)$ and therefore $\iota(g)\in H$.

  We next show that $H$ is normal in $B^X$. For this, consider $h\in
  H$. It suffices to show that $h^{\iota(b_i)}$ belongs to $H$ for all
  $i$. Now $h^{\iota(b_i)}=h^{f^{g_i}}$ belongs to $H$, and we are
  done.
\end{proof}

\begin{proposition}[\cite{bartholdi-erschler:imbeddings}*{Proposition~6.2}]\label{prop:Wconvergence}
  Let $G$ be a group acting on $X$. Let the sequence $(x_i)$ in $X$ be
  spreading and locally stabilizing. Let a sequence of elements
  $(b_i)$ be given in the group $B$, all of the same order
  $\in\N\cup\{\infty\}$.

  For all $i\in\N$, let $f_i$ be the finitely supported function $X\to
  B$ with $f_i(x_{n(j)})=b_j$ for all $j\le i$, all other values
  being trivial, and denote by $W_i$ the group $\langle f_i,G\rangle$.

  Then for every increasing sequence $(m(i))$ there is a choice of
  $(n(i))$ such that the ball of radius $m(i)$ in $W$ coincides with
  the ball of radius $m(i)$ in $W_i$, via the identification
  $f\leftrightarrow f_i$.

  Furthermore, the term $n(i)$ depends only on $m(i)$ and on the ball
  of radius $m(i)$ in $\langle b_1,\dots,b_{i-1}\rangle$.
\end{proposition}
\begin{proof}
  Choose $n(i)$ such that $d(x_j,x_k)\ge m(i)$ for all $j\neq k$ with
  $k\ge n(i)$, and such that the balls of radius $m(i)$ around
  $x_{n(i)}$ and $x_j$ coincide for all $j>n(i)$.

  Consider then an element $h\in W$ in the ball of radius $m(i)$, and
  write it in the form $h=(c,g)$ with $c\colon X\to B$ and $g\in
  G$. The function $c$ is a product of conjugates of $f$ by words of
  length $<R$. Its support is therefore contained in the union of
  balls of radius $m(i)-1$ around the $x_j$, with $j$ either $\ge
  n(i)$ or of the form $n(k)$ for $k<i$. In particular, the entries of
  $c$ are in $\langle b_1,\dots,b_{i-1}\rangle\cup\bigcup_{j\ge
    i}\langle b_j\rangle$. For $j>n(i)$, the restriction of $c$ to the
  ball around $x_j$ is determined by the restriction of $c$ to the
  ball around $x_{n(i)}$, via the identification $b_i\mapsto b_j$,
  because the neighbourhoods in $X$ coincide and all cyclic groups
  $\langle b_j\rangle$ are isomorphic.

  It follows that the element $h\in W$ is uniquely determined by the
  corresponding element in $W_i$.
\end{proof}

\begin{corollary}[\cite{bartholdi-erschler:imbeddings}*{Corollary~6.3}]\label{cor:subexp}
  Let $G$ be a group acting on $X$ with the subexponential wreathing
  property. Let the sequence $(x_i)$ be spreading and locally
  stabilizing.

  If $B$ has locally subexponential growth, then $W$ has
  subexponential growth.
\end{corollary}
\begin{proof}
  Let $Z=\langle z\rangle$ be a cyclic group whose order (possibly
  $\infty$) is divisible by the order of the $b_i$'s.  We replace $B$
  by $B\times Z$ and each $b_i$ by $b_i z$, so as to guarantee that
  all generators in $B$ have the same order.

  Let $\epsilon_i$ be a decreasing sequence tending to $1$. Denote by
  $v_i$ the growth function of the group $W_i$ introduced in
  Proposition~\ref{prop:Wconvergence}, and by $w$ the growth function
  of $W$. Let $m(i)$ be such that
  \[v_i(m(i))\le \epsilon_i^{m(i)}.
  \]
  Such an $m(i)$ exists, because $B\wr_X G$ has locally subexponential
  growth. Since the balls of radius $m(i)$ coincide in $W$ and $W_i$,
  we also have $w(m(i))\le\epsilon_i^{m(i)}$. Then, if $R>m(i)$, we get
  \[w(R)\le\epsilon_i^{R+m(i)},\] so
  $\lim\sqrt[R]{w(R)}\le\epsilon_i$. Since this holds for all $i$, the
  growth of $W$ is subexponential.
\end{proof}

\section{Imbedding a group in a derived subgroup}\label{ss:derived}
Let $G=\langle S\rangle$ be a group with fixed generating set $S$. We
denote by $\|\cdot\|_S$ the word norm on $G$, or $\|\cdot\|_G$ if the
generating set is clear. Let us now define another norm on
$[G,G]$. For this, let us say that a word $w$ in the free group $F_S$
is \emph{balanced} if it belongs to $[F_S,F_S]$; namely, if it
contains as many $s$'s as $s^{-1}$'s for every letter $s\in S$. The
\emph{perfect norm} on $[G,G]$ is
\[\|g\|_{\text{perfect}}=\min\{\|w\|\colon w\in F_S\text{ is a perfect
  word representing }g\}.
\]
We denote by $d_S(x,y)=d_G(x,y)=\|x y^{-1}\|_S$ and
$d_{\text{perfect}}(x,y)=\|x y^{-1}\|_{\text{perfect}}$ the
corresponding distances.

\begin{proposition}\label{prop:imbed[G,G]}
  Let $G=\langle S\rangle$ be a finite group with fixed generating set
  of cardinality $d$. Then there exists a finite group $H=\langle
  T\rangle$ with generating set of cardinality $d+1$ and an imbedding
  $\iota\colon G\to[H,H]$ such that
  \[2\|g\|_G\le\|\iota(g)\|_{perfect}\le 4\|g\|_G\text{ for all }g\in G.\]
\end{proposition}
\begin{proof}
  Let $m$ denote the cardinality of $G$. Since $G*\Z$ is residually
  finite, there exists a finite quotient $Q=\langle S\cup\{x\}\rangle$
  of $G*\Z$ such that the balls of radius $m$ in $G*\Z$ and in $Q$
  coincide.

  In the group $Q\wr C_{2m}$, consider the following elements: for
  every $s\in S$, the function $t_s\colon C_{2m}\to Q$ with values
  $(1,s,s^2,\dots,s^{m-1},1,s^x,s^{2x},\dots,s^{(m-1)x})$; and the
  generator $r$ of $C_{2m}$. Set
  \[T=\{t_s:s\in S\}\cup\{r\}\text{ and }H=\langle T\rangle.\]
  Define $\iota\colon G\to Q\wr C_{2m}$ by
  \[\iota(g)\colon C_{2m}\to Q\text{ taking values }(g,\dots,g,g^x,\dots,g^x).\]

  Note first that, for $s\in S$, we have $\iota(t)=[t_s,r]$. This
  immediately implies $\iota(G)\le[H,H]$ and gives the inequality
  $\|\iota(h)\|_{perfect}\le 4\|h\|_G$.

  Note then that if a word of length $\le m/2$ in $T$ is not of the
  form $q_0r^{-1}q_1 r\cdots q_{2n-2}r^{-1}q_{2n-1}r q_{2n}$ with all
  $q_i\in\langle t_s:s\in S\rangle$, then it cannot belong to the
  image of $\iota$. On the other hand, if it is of that form, then
  write it as $f\colon C_{2m}\to Q$, and note that $f(1)$ depends only
  on $q_1,\dots,q_{2n-1}$ and has length at most the sum of their
  lengths. This gives the other inequality
  $2\|h\|_G\le\|\iota(h)\|_{perfect}$.
\end{proof}

\section{Imbedding a sequence of groups in $W$}\label{ss:inW}
We now apply the construction of Section~\ref{ss:W} to a restricted
direct product of finite groups.  The heart of the argument is the
following variant of Corollary~\ref{cor:subexp}:
\begin{proposition}\label{prop:W_S}
  Let $(H_i)_{i\in\N}$ be a sequence of $d$-generated finite groups.

  Then there exists a family of groups $(W_S)_{S\subset\N}$ indexed by
  subsets $S$ of $\N$, each of subexponential growth, with the
  following property: for all $s\in S$, there is an imbedding
  $\Psi_s\colon [H_s,H_s]\to W_S$ that is $(K,L)$-bi-Lipschitz with
  respect to the perfect metric of $[H_s,H_s]$ and the word metric on
  $W_S$, and such that the constants $K,L$ depend only on $\{H_i\colon
  i<s\}$.

  Furthermore, if $S=\{s(1),s(2),\dots\}$, then
  $W_{\{s(1),\dots,s(i+1)\}}$ is constructed out of
  $W_{\{s(1),\dots,s(i)\}}$, and the sequence
  $(W_{\{s(1),\dots,s(i)\}})_{i\in\N}$ converges to $W_S$ in the
  Cayley topology.
\end{proposition}
\begin{proof}
  Up to replacing $(H_i)_{i\in\N}$ by $(H_i)_{i\in S}$, we lighten
  notation and suppose $S=\N$ or a prefix $\{1,2,\dots,n\}$ thereof.

  Let us write $T_i=\{t_{i,1},\dots,t_{i,d}\}$ the generating set of
  $H_i$. We consider the restricted direct product
  $B={\prod_{i\ge1}}'H_i$.  It is a countable, locally finite group,
  generated by
  $\{t_{1,1},\dots,t_{1,d},t_{2,1},\dots\}=\{b_1,b_2,\dots\}$. We
  consider the group $W$ constructed in Section~\ref{ss:W}, noting
  that the sequence $n(i)$ may be chosen, by
  Corollary~\ref{cor:subexp}, such that $W$ has subexponential growth,
  and that $n(i)$ depends only on $H_1,H_2,\dots,H_{\lceil
    i/d\rceil-1}$.

  Consider $[H_i,H_i]$ as a subgroup of $W$, imbedded as the functions
  $X\to H_i\subset B$ supported only at $x_{n(di)}$. This is an
  imbedding by Lemma~\ref{lem:contains [B,B]}. Denote by $\Psi_i\colon
  H_i\to W$ this imbedding.

  Assume that $n(1),\dots n(di)$ have already been chosen; and note
  that their choice relies only on $H_1,\dots,H_{i-1}$. Recall also
  that $f(x_{n(di-d+j)})=t_{i,j}$ in the construction of $W$. We now
  show that there exist constants $K,L$ independent of $H_i$ such that
  the imbedding $\Psi_i\colon[H_i,H_i]\to W$ is
  $(K,L)$-bi-Lipschitz. In other words, independently of $i$, the
  distortion of $[H_i,H_i]$ in $W$ is at worst $\rho(t)=tK/L$.

  Let $g_1,\dots,g_d\in G$ be such that $x_{n(di-d+j)}g_j = x_{n(di)}$
  and the only $x_k$ mapped to another $x_\ell$ under $g_{j'}g_j^{-1}$
  are $x_{n(di-d+j')}g_{j'}g_j^{-1}=x_{n(di-d+j)}$; such elements
  exist because $(x_i)$ is separating. Let $L'$ be an upper bound for
  the lengths of all $g_1,\dots,g_d$. This condition ensures that the
  functions $f^{g_1},\dots,f^{g_d}$ have disjoint support except at
  $x_{n(di)}$ or where they coincide.

  Here is an explicit way of computing the imbedding $\Psi_i\colon
  [H_i,H_i]\to W$: for $h\in [H_i,H_i]$, write it as a minimal-length
  balanced word in $T_i$, and map each letter $t_{i,j}$ to $f^{g_j}$.

  On the one hand, $\|\Psi_i(h)\|_W\le(2L'+1)\|h\|_{\text{perfect}}$
  because each letter gets mapped to a word of length
  $1+2\|g_j\|\le1+2L'$; on the other hand, $\|\Psi_i(h)\|_W\ge\|h\|$,
  because at most one element of $T$ is contributed by each generator
  of $W$.
\end{proof}

\begin{corollary}
  Let $(G_i)_{i\in\N}$ be a sequence of $d$-generated finite groups.

  Then there exists a family of groups $(W_S)_{S\subset\N}$ indexed by
  subsets $S$ of $\N$, each of subexponential growth, with the
  following property: for all $s\in S$, there is an imbedding
  $\Psi_s\colon G_s\to W_S$ that is $(K,L)$-bi-Lipschitz with respect
  to the word metrics, and such that the constants $K,L$ depend only
  on $\{G_i\colon i<s\}$.
\end{corollary}
\begin{proof}
  Using Proposition~\ref{prop:imbed[G,G]}, imbed each $G_i$ in a
  $(d+1)$-generated group $H_i$ in such a manner that the inclusion
  map $\iota_i\colon (G_i,d_{G_i})\to (H_i,d_{\text{perfect}})$ is
  $(2,4)$-bi-Lipschitz. Apply then Proposition~\ref{prop:W_S} to the
  family $(H_i)_{i\in\N}$.
\end{proof}

\section{Imbeddings into metric spaces}\label{ss:inMS}
Let $\mathcal X$ be a metric space. Given a sequence of metric spaces
such as $((G_i,d_{G_i}))_{i\in\N}$, we say that it \emph{imbeds \coarse ly}
in $\mathcal X$ if there exists an unbounded increasing function
$\rho$ and a sequence of $1$-Lipschitz imbeddings $(\Phi_i\colon
G_i\to\mathcal X)$, each with distortion better than $\rho$. We are
interested in the opposite property:
\begin{definition}\label{defn:coarseimbed}
  Let $\mathcal X$ be a metric space. A sequence of metric spaces
  $((G_i,d_{G_i}))_{i\in\N}$ \emph{does not imbed \coarse ly} in $\mathcal
  X$ if the following holds: there exists a constant $M$ such that, if
  $(\Phi_i\colon G_i\to\mathcal X)$ is a sequence of $1$-Lipschitz
  imbeddings, then, for all $t\in\R$, there are $i\in\N$ and $x,y\in
  G_i$ with $d(x,y)\ge t$ and $d(\Phi_i(x),\Phi_i(y))\le M$.
\end{definition}

\begin{theorem}\label{thm:Edistortion}
  Let $(G_i)_{i\in\N}$ be a sequence of $d$-generated finite
  groups. Assume furthermore that $(G_i)$ does not imbed \coarse ly in
  a metric space $\mathcal X$. Let $\rho$ be any unbounded increasing
  function $\R_+\to\R_+$. Then there exists a finitely generated group
  $W$ of subexponential growth such that every imbedding of $W$ in
  $\mathcal X$ has distortion worse than $\rho$.
\end{theorem}
\begin{proof}
  Using Proposition~\ref{prop:imbed[G,G]}, imbed each $G_i$ in a group
  $H_i$ whose derived subgroup, with the perfect metric, contains a
  bi-Lipschitz copy of $G_i$. We identify $G_i$ with its image in $H_i$.

  Since the groups $G_i$ do not imbed \coarse ly in $\mathcal X$,
  there exists a constant $M$ such that, if $(\Phi_i\colon
  G_i\to\mathcal X)$ is a sequence of $1$-Lipschitz imbeddings then,
  for all $t\in\R$, there are $i\in\N$ and $x,y\in G_i$ with
  $d(x,y)\ge t$ and $d(\Phi_i(x),\Phi_i(y))\le M$.

  The group $W$ will be of the form $W=W_S\times W_{S'}$, for
  sequences $S=\{s(1),s(3),s(5),\dots\}$ and $S'=\{s(2),s(4),\dots\}$
  that we construct iteratively by applying Proposition~\ref{prop:W_S}
  to the family $(H_i)$. Assume that $s(1),\dots,s(i-1)$ have been
  constructed. The terms $s=s(i)$ and $s'=s(i+1)$ have not yet been
  determined, but we already know the constants
  $K_i,L_i,K_{i+1},L_{i+1}$ such that the imbedding of $G_s$ into
  $W_S$ or $W_{S'}$ will be $(K_i,L_i)$-bi-Lipschitz and the imbedding
  of $G_{s'}$ into $W_{S'}$ or $W_S$ will be
  $(K_{i+1},L_{i+1})$-bi-Lipschitz.

  We now make use of the unbounded function $\rho$. Let $t_i\in\R$ be
  large enough so that $\rho(t_i)>L_{i+1}M$. Since the $(G_i)$ do not
  imbed \coarse ly, we can choose $s$ large enough so that there exist
  $x,y\in G_s$ with $d(x,y)\ge t_i/K_i$ and $d(\Phi_s(x),\Phi_s(y))\le
  M$ in any $1$-Lipschitz imbedding $\Phi_s$ of $G_s$ into $\mathcal
  X$. Without loss of generality, the sequences $(t_i)$ and $(s(i))$
  are strictly increasing. This determines $s=s(i)$, and finishes the
  inductive construction of $S$ and $S'$.

  Let us now check that the group $W$ just constructed has the desired
  property. Let $\Phi\colon W\to\mathcal X$ be a $1$-Lipschitz
  imbedding. By composing with the imbedding of $G_s$ in $W_S$ or
  $W_{S'}$, we get for all $s\in S\cup S'$ imbeddings
  $\Phi_s=\Phi\circ\Psi_s$ of $G_s$ into $\mathcal X$.

  Consider $t\in\R_+$, and suppose $t\ge t_1$. Let $i$ be such that
  $t_{i-1}\le t<t_i$. Set $s=s(i)$.  Following the construction above,
  the imbedding $\Phi_s$ is $(K_i,L_i)$-Lipschitz, so there are
  $x,y\in G_s$ with $d(x,y)\ge t_i/K_i$ so
  $d(\Psi_s(x),\Psi_s(y))\ge t_i$ while $d(\Phi_s(x),\Phi_s(y))\le L_i
  M$. This proves that the distortion $\rho_\Phi$ of $\Phi$
  satisfies
  \[\rho_\Phi(t)\le\rho_\Phi(t_i)\le L_i M<\rho(t_{i-1})\le\rho(t),\]
  so the distortion of $W$ is worse than $\rho$.
\end{proof}
We note from the proof that the distortion of a single copy $W_S$ is
worse than $\rho$ along an unbounded sequence.

\section{Superexpanders}\label{ss:superexpanders}
We now exhibit a sequence of finite groups $(H_i)$ with particularly
bad imbedding properties and fixed-size generating set. We recall the
following definition from~\cite{lafforgue:Tbanachique}.

A Banach space $\mathcal X$ is called \emph{of type $>1$}, or
\emph{B-convex}, if there are $n\in\N$ and $\epsilon>0$ such that no
$1$-Lipschitz imbedding $(\mathbb C^n,\|\cdot\|_1)\to E$ with
distortion better than $t\mapsto t/(1+\epsilon)$ exists. For example,
Hilbert space is B-convex with $n=2$ and any
$\epsilon>\sqrt2-1$. Lafforgue gives the following construction of
expanders:
\begin{proposition}[\cite{lafforgue:Tbanachique}*{Corollaire~0.5}]\label{prop:lafforgue}
  There exists a sequence $(H_i)_{i\in\N}$ of finite quotients of a
  finitely generated group $H$ such that the sequence of quotient
  Cayley graphs of the $H_i$ does not admit any \coarse\ imbedding
  into a B-convex Banach space.
\end{proposition}
\begin{proof}
  What Lafforgue shows, actually, is that there exists a constant $C$
  such that, for every $i\in\N$ and every $1$-Lipschitz map
  $\Phi\colon H_i\to\mathcal X$ with $0$ mean, one has $\mathbb
  E_{x\in H_i}\|\Phi(x)\|^2\le C$. A classical argument (see, e.g.,
  \cite{lafforgue:T}*{page~600}) implies that there are two points
  $x,y\in H_i$ with $d(x,y)\ge c\log(\#H_i)$ and
  $d(\Phi(x),\Phi(y))\le\sqrt{2C}$, for a constant $c$ depending only
  on the number of generators of $H$.
\end{proof}

Here is a concrete example: consider a prime power $q$, the group
$H=\mathbf{SL}_3(\mathbb F_q[t])$, and its images $H_i$ in
$\mathbf{SL}_3(\mathbb F_q[t]/(t^i))$. In this situation, we have a
few extra, useful properties, which we quote as a
\begin{lemma}
  Additionally, in Proposition~\ref{prop:lafforgue}, the group $H$ may
  be supposed to be perfect.
\end{lemma}
\begin{proof}
  Since $\mathbb F_q[t]$ is a Euclidean domain, $H$ is generated by
  elementary matrices. Furthermore, the classical identities
  $X_{i,j}(P+Q)= X_{i,j}(P)X_{i,j}(Q)$ and
  $X_{i,j}(PQ)=[X_{i,k}(P),X_{k,j}(Q)]$ between elementary matrices,
  when $\{i,j,k\}=\{1,2,3\}$, imply that $H$ is generated by
  $A=\mathbf{SL}_3(\mathbb F_q)$ and $B=\langle
  X_{1,2}(t)\rangle$. Since $A$ is perfect and $B^A=\{1+tM\colon M\in
  M_3(\mathbb F_q)\text{ and }\operatorname{tr}(M)=0\}$ has no $A$-invariant
  element, $H$ is also perfect.
\end{proof}
We fix as generating set $T=A\cup B$, and denote by $T_i$ its natural
image in $H_i$.

\begin{lemma}\label{lem:H1H2H3}
  Let $H$ be a finitely generated perfect group, and let
  $(H_i)_{i\in\N}$ be a family of quotients of $H$. Then the groups
  $H_i$ are perfect, all generated by the same number of elements, and
  the identity map $(H_i,d_{\text{perfect}})\to(H_i,d_{H_i})$ is
  $(J^{-1},1)$-bi-Lipschitz for a constant $J\ge1$ independent of $i$.
\end{lemma}
\begin{proof}
  The abelianisations of the $H_i$ are quotients of the abelianisation
  of $H$, and therefore are trivial.  If $H=\langle T\rangle$ be
  $d$-generated, then the groups $H_i$ are naturally $d$-generated by
  the images $T_i$ of $T$.

  Represent now each $t\in T$ as a balanced word, and let $J$ be the
  maximal length of these balanced words. Then
  \[\|g\|_{T_i}\le \|g\|_{\text{perfect}}\le J\|g\|_T\le
  J\|g\|_{T_i}\text{ for all }g\in [H_i,H_i],
  \]
  so the inclusion $([H_i,H_i],d_{\text{perfect}})\to(H_i,d_{T_i})$ is
  $(J^{-1},1)$-bi-Lipschitz.
\end{proof}

\begin{corollary}\label{cor:distortion}
  Let $\mathcal X$ be a B-convex Banach space, and let $\rho$ be an
  unbounded increasing function $\R_+\to\R_+$. Then there exists a
  finitely generated group $W$ of subexponential growth such that
  every imbedding of $W$ in $\mathcal X$ has distortion worse than
  $\rho$.

  In particular, let $\rho$ be any unbounded increasing function
  $\R_+\to\R_+$. Then there exists a finitely generated group $W$ of
  subexponential growth such that every imbedding of $W$ into Hilbert
  space has distortion worse than $\rho$.
\end{corollary}
\begin{proof}
  Consider the sequence of superexpanders $(H_i)_{i\in\N}$ given by
  Proposition~\ref{prop:lafforgue}. Either imbed them as derived
  subgroups in finite groups, using Proposition~\ref{prop:imbed[G,G]},
  or note that that step is unnecessary, thanks to
  Lemma~\ref{lem:H1H2H3}.

  By Proposition~\ref{prop:lafforgue}, there exists a constant $M$
  such that if $(\Phi_i\colon H_i\to \mathcal X)$ is a sequence of
  $1$-Lipschitz imbeddings into $\mathcal X$ then for every $t\in\R$
  and for all $j$ large enough (depending on $t$) there exist $x,y\in
  H_j$ with $d(x,y)\ge t$ and $d_{\mathcal X}(\Phi_j(x),\Phi_j(y))\le
  M$.

  Theorem~\ref{thm:Edistortion} then applies.
\end{proof}

\begin{remark}\label{rem:switchquantifiers}
  The order of quantifiers can be switched in
  Corollary~\ref{cor:distortion}: fixing $n\in\N$ and $\epsilon>0$,
  there exists for every unbounded increasing function $\R_+\to\R_+$ a
  group $W$ of subexponential growth with the following property: if
  $\mathcal X$ is any B-convex Banach space not
  $(1+\epsilon)$-isometrically containing $\ell_1(\mathbb C^n)$, then
  every imbedding of $W$ in $\mathcal X$ has distortion worse than
  $\rho$. This follows formally from Corollary~\ref{cor:distortion}
  because an $\ell^2$-sum of such Banach spaces is again of the same
  form.
\end{remark}


\end{document}